\documentclass[11pt]{article}
\usepackage{amsfonts,amssymb,amsmath,amsthm,amscd}
\newtheorem{prop}{Proposition}[section]
\newtheorem{lem}[prop]{Lemma}
\newtheorem{thm}[prop]{Theorem}
\newtheorem{cor}[prop]{Corollary}

\newtheorem{remar}[prop]{Remark}
\newtheorem{exe}[prop]{Example}

\begin{document}

\title{Elliptic curves of bounded degree \\ in a polarized Abelian variety}
\author{Lucio Guerra}
\date{}
\maketitle

\begin{abstract} \noindent
For a polarized complex Abelian variety $A$, of dimension $g>1$,
we study the function $N_A(t)$ counting the number of elliptic curves in $A$
with degree bounded by $t$. We describe elliptic curves as solutions of 
Diophantine equations which, at least for small dimensions $g=2$ and $g=3$,
can actually be made explicit, and we show that computing
the number of solutions is reduced to the classical topic 
in Number Theory of counting points of the lattice $\mathbb Z^n$ 
lying on an explicit bounded subset of $\mathbb R^n$.
We obtain, for Abelian varieties of small dimension,
some upper bounds for the counting function.
\medskip

\noindent {\sc M.S.C. \  14K20 \  11D45}
\end{abstract}

\section*{Introduction}

Let $A$ be a complex torus, of dimension $g>1$.
Elliptic curves in $A$ are translates of one-dimensional subtori. 
We choose the following terminology: when we use the expression
'elliptic curve in a complex torus' we always mean a one-dimensional subtorus.
The collection of all elliptic curves in $A$ is (at most) countable
(and possibly empty). Assume that $A$ is a complex Abelian variety, endowed with a
polarization. Every algebraic curve in $A$ has a degree with respect to the polarization, 
and the following finiteness theorem holds:
{for every integer $t \geq 1$ the collection of elliptic curves $E \subset A$
such that $\deg(E) \leq t$ is finite.}

In dimension $g=2$ this was known to Bolza and Poincar\'e, and
a modern account is in the paper of Kani \cite{K2}. For Jacobian varieties $J(C)$
of arbitrary dimension the theorem was proved by Tamme and was brought to
an effective form in another paper of Kani \cite{K1}. 
For an arbitrary Abelian variety $A$ the theorem follows from a
general result proved by Birkenhake and Lange in \cite{BL},
to the effect that the collection of all Abelian subvarieties with
bounded exponent in $A$ is finite.

Denote by $N_A(t)$ the number of elliptic curves in $A$
with degree bounded by $t$. The aim in the present paper is to study
the function $N_A(t)$, for small dimensions $g=2$ and $g=3$.
The problem of bounding this function is invariant under isogenies,
and the most relevant case is when $A$
is the self product $E^g$ of an elliptic curve, and the polarization
is the product of a sequence of pullback polarizations from the factors.
When we consider $E^g$ as a polarized Abelian variety 
we always assume that it is endowed with such a product polarization.

There is always in $E^g$ an infinite collection of elliptic curves, that we call
{\em ordinary}. We show that there are {\em extra-ordinary} elliptic curves 
(in fact infinitely many) if and only if $E$ admits complex multiplication.
These results are given in several steps in \S \ref{Eg}.
We show that computing elliptic curves in $E^g$ is reduced to solving
certain Diophantine (systems of) equations, in terms of coordinates in
$H_2(E^g, \mathbb Z)$, and the main achievement in the present paper
is that these equations can actually be made explicit
(Theorem \ref{g=2} for $g=2$, Theorem \ref{g=3} for $g=3$).

It turns out that computing $N_{E^g}(t)$ is reduced to 
counting points of the lattice $\mathbb Z^n$ 
lying on an explicit bounded subset of $\mathbb R^n$, for some $n$ depending on $g$.
This is a classical topic in Number Theory, originating from Gauss' circle problem
and still a field of active research. So we are lead to apply some results from that field,
and in this way we can give some explicit upper bounds for the function $N_{E^g}(t)$
(Proposition \ref{Ndim2} for $g=2$, Proposition \ref{Ndim3} for $g=3$).
Finally the results for a self product easily imply results for an arbitrary 
polarized Abelian variety of small dimension, and we remark that 
the estimates for $N_A(t)$ obtained in this way are, at least in the rate of growth, 
sharper than certain existing bounds (see \S \ref{Ag=23}).

\section{Some preliminary material}

\subsection{Elliptic curves and homology classes} \label{homology}

Let $A$ be an Abelian variety. There is a natural isomorphism
$$H_1(A,\mathbb Z) \wedge H_1(A,\mathbb Z) 
\overset{_\sim}{\longrightarrow} H_2(A,\mathbb Z)$$
such that $\lambda\wedge\mu \mapsto \lambda\star\mu$
where $\star$ is the Pontryagin product of 1-cycles
(cf. \cite{CAV}, either \S 1.5 (8) or else Lemma 4.10.1).

Let $E$ be an elliptic curve in $A$ (a one-dimensional subtorus).
It turns out that a basis of $H_1(E,\mathbb Z)$ is symplectic,
for the principal polarization or equivalently for every polarization, if and only if is
negatively oriented, with respect to the canonical orientation of the complex
vector space $H_1(E,\mathbb C)$ (cf. \cite{CAV}, Lemma 3.6.5).

If $\lambda,\mu$ form a symplectic basis in $H_1(E,\mathbb Z)$,
the fundamental class of $E$ in $H_2(E, \mathbb Z)$ is given by
$\{ E \} = - \lambda \star \mu$  (cf. \cite{CAV}, Cor. 4.10.5)
and by functoriality the same holds in $H_2(A, \mathbb Z)$.
Therefore, if $\lambda,\mu$ form a positively oriented basis, 
then the fundamental class is  $$\{ E \} = \lambda \star \mu.$$

The correspondence sending an elliptic curve $E \subset A$
to the homology class $\{ E \} \in H_2(A, \mathbb Z)$ is injective.
This is certainly well known, however a proof is also given later on
in Remark \ref{injective}.

\subsection{Degree with respect to the polarization} \label{degpol}

Let $A$ be a polarized Abelian variety. 
The element of $H^2(A,\mathbb Z)$ corresponding to the polarization,
viewed as a bilinear antisymmetric form on $H_1(A,\mathbb Z)$,
will be written as a bracket $(-,-)$. 
Let $\Theta$ be a divisor in the associated linear system. 

For an elliptic curve $E \!\subset\! A$ the degree with respect to the polarization 
is $$\deg(E) = E \cdot \Theta.$$ 
Assume that $\lambda,\mu$ form a symplectic basis in $H_1(E,\mathbb Z)$. 
The degree of $E$ is also described as  $$E \cdot \Theta = (\lambda, \mu)$$
and may therefore be seen as the degree of the induced polarization on $E$.

This follows from the equality
$(\lambda \star \mu) \cdot \Theta = - (\lambda, \mu).$
If, in the inclusion $H_1(E,\mathbb Z) \rightarrow H_1(A,\mathbb Z)$, the elements
$\lambda,\mu$ become part of a symplectic basis of $H_1(A,\mathbb Z)$,
this is found in \cite{CAV}, Lemma 4.10.3, and in general it follows by bilinearity.

\subsection{Behaviour under isogenies} \label{isogenies}

Let $A,B$ be polarized Abelian varieties and let $B \rightarrow A$ be an isogeny, 
preserving the polarizations, whose degree we call $d$.
There is a one to one correspondence between elliptic curves in $A$ and in $B$.

Given $E \subset A$ the corresponding $E^\ast$ in $B$ is
the connected component of $0$ in the pre-image of $E$. 
The restricted isogeny $E^\ast \rightarrow E$ has degree $d_E \leq d$
(in fact a divisor of $d$), and the degree of $E^\ast$ is given by
\begin{equation*}
\deg(E^\ast) = d_E\, \deg(E)
\end{equation*}
(by the projection formula:
$E^\ast \cdot f^\ast \Theta = 
f_\ast E^\ast \cdot \Theta = d_E\; E \cdot \Theta$).
Therefore:
\begin{equation*}
\deg(E) \leq \deg(E^\ast) \leq d\, \deg(E).
\end{equation*}

It follows that the functions counting elliptic curves in $A$ and in $B$
are related by the following inequalities:
$$N_A(t) \leq N_B(dt) {\rm \ \ \ and \ \ \ } N_B(t) \leq  N_A(t),$$
and this has the following remarkable consequence:
the property that every collection of elliptic curves of bounded degree in $A$ is finite and
the property that the counting function $N_A(t)$ is asymptotically of type $O(t^\alpha)$
are invariant under polarization preserving isogenies of $A$.

\subsection{The finiteness theorem} \label{finthm}

Let $A$ be a polarized Abelian variety, of dimension $g>1$. \bigskip

\noindent {\bf Theorem.}
{\em For every integer $t \geq 1$ the collection of elliptic curves $E \subset A$
such that $\deg(E) \leq t$ is finite.} \bigskip

This is a consequence of a general result proved by
Birkenhake and Lange \linebreak in \cite{BL}, to the effect that 
the collection of all Abelian subvarieties with bounded exponent in $A$ is finite. 
Let us recall some definitions. 
The polarization defines a natural isogeny $\phi: A \rightarrow \widehat A$ 
to the dual variety. The order of ${\rm ker}(\phi)$ is the degree of the polarization
and the exponent of ${\rm ker}(\phi)$ is called the exponent of the polarization on $A$.
Clearly the exponent is a divisor of the degree.
For an Abelian subvariety $E$ of $A$ one has the exponent and the degree of the
induced polarization. 
If $E$ is an elliptic curve in $A$ we know that the
degree of the curve is equal to the degree of the induced polarization.
So elliptic curves with bounded degree have bounded exponent, and the theorem follows.

Assume now that $A = J(C)$ is the Jacobian variety of a curve
of genus $g>1$, with the canonical polarization.

In this case, elliptic curves in $J(C)$ correspond bijectively to isomorphism classes of
surjective morphisms $C \rightarrow E$, with $E$ an elliptic curve,
which do not factor as $C \rightarrow E' \rightarrow E$, where
$E' \rightarrow E$ is surjective of positive degree.
(Here two morphisms from $C$ to $E$ and to $E'$ are said to be isomorphic
if they are related by some isomorphism $E' \rightarrow E$.)
Associated to $f:C \rightarrow E$ is the pull-back homomorphism
$f^\ast: E \cong {\rm Pic}^0(E) \longrightarrow {\rm Pic}^0(C) \cong J(C)$
and the image $f^\ast(E)$ is the corresponding elliptic curve in $J(C)$.
The degree of this curve is $$f^\ast(E) \cdot \Theta = \deg (f).$$
In fact, one has $f^\ast(E) \cdot \Theta = \deg\, (f^\ast)^\ast (\Theta)$
and, since $(f^\ast)^\ast (\Theta)$ is analytically equivalent to $\deg (f)\, \Theta'$
where $\Theta'$ is a theta divisor of ${\rm Pic}^0(E)$ (cf. \cite{CAV}, Lemma 12.3.1),
then $\deg\, (f^\ast)^\ast (\Theta) = \deg (f)$.

Therefore, as a corollary of the theorem above it follows that: 
for every integer $t \geq 1$ the collection of 
isomorphism classes of morphisms $f: C \rightarrow E$ which do not factor 
through a non-trivial covering of $E$ and have $\deg (f) \leq t$ is finite.
This was first proved by Tamme and was brought to an effective form
in the paper of Kani \cite{K1}.

\section{Elliptic curves in a complex torus}

Let $\Lambda$ be a free module of finite rank over $\mathbb Z$.
A rank 2 submodule can be presented giving a basis $\lambda,\mu$.
If the submodule has to be maximal among rank 2 submodules,
we may assume that $\lambda$ is a primitive element.
In this setting, it is easy to see that the following are equivalent conditions:
\begin{itemize}
\item[$-$]
$\mu$ is primitive mod $\lambda$, i.e. represents a primitive
element in $\Lambda / \mathbb Z \lambda$;
\item[$-$]
every $\mu + t \lambda$ is primitive in $\Lambda$;
\item[$-$]
$\langle \lambda,\mu \rangle$ is maximal among
rank two submodules of $\Lambda$.
\end{itemize}
In general, every nonzero element in a free module over $\mathbb Z$ has
an {\em integral content}, that by definition is in $\mathbb Z_{>0}$. 

\begin{lem} \label{content}
In the setting above, the content of $[\mu]$ in $\Lambda / \mathbb Z \lambda$ 
and the content of $\lambda \wedge \mu$ in $\Lambda \wedge \Lambda$ coincide.
\end{lem}

\begin{proof}
It is enough to prove that for an integer $r$ one has that:
$r$ divides $[\mu]$ in $\Lambda / \mathbb Z \lambda$ if and only if
$r$ divides $\lambda \wedge \mu$ in $\Lambda \wedge \Lambda$.

If $\mu = u \lambda + r \mu'$ then $\lambda \wedge \mu = r\, \lambda \wedge \mu'$. 
Conversely, assume that $\lambda \wedge \mu = r \omega$. 
Write $r = p s$ with $p$ prime. Note that
$(\Lambda \wedge \Lambda) / p (\Lambda \wedge \Lambda)
\cong (\Lambda /p \Lambda) \wedge (\Lambda /p \Lambda)$ 
is a vector space over $\mathbb Z / p \mathbb Z$. Here one has
$\lambda \wedge \mu \equiv 0$ and since $\lambda \not\equiv 0$
in $\Lambda /p \Lambda$ it follows that $\mu = c \lambda + p \nu$.
Thus $\lambda \wedge \mu = p\, \lambda \wedge \nu$ with
$\lambda \wedge \nu = s \omega$.
Then write $s = p' t$ with $p'$ prime. By the same argument we have that
$\nu = c' \lambda + p' \nu'$ and $\lambda \wedge \nu' = t \omega$, and
hence $\mu = c \lambda + p (c'\lambda + p'\nu') = (c+pc')\lambda + pp' \nu'$.
Iterating on the sequence of prime factors of $r$ one ends with 
$\mu =  u\lambda + r \mu'$ for some $u$ and $\mu'$
(and note, by the way, that $\lambda \wedge \mu' = \omega$).
\end{proof}

\begin{cor} \label{primitive}
In the setting above, the following are equivalent conditions:
\begin{itemize}
\item[$-$]
$\mu$ is primitive mod $\lambda$;
\item[$-$]
$\lambda \wedge \mu$ is primitive in $\Lambda \wedge \Lambda$.
\end{itemize}
\end{cor}

Assume now that $\Lambda$ is a lattice in $\mathbb C^g$. 
There is a natural $\mathbb Z$-linear homomorphism
$\Lambda \wedge \Lambda \rightarrow \mathbb C^g \wedge \mathbb C^g$.
We characterize pairs $\lambda,\mu$ which generate the submodule of $\Lambda$
corresponding to an elliptic curve in the complex torus $\mathbb C^g / \Lambda$.

\begin{prop} \label{ellipticsublattice}
In the setting above, $\lambda,\mu$ is a basis for the lattice of an elliptic curve
in the complex torus if and only if:
\begin{itemize}
\item[$-$]
$\lambda \wedge \mu$ maps to $0$ in the homomorphism
$\Lambda \wedge \Lambda \rightarrow \mathbb C^g \wedge \mathbb C^g$, and
\item[$-$]
$\lambda \wedge \mu$ is primitive in $\Lambda \wedge \Lambda$.
\end{itemize}
\end{prop}

\begin{proof}
The first property means that $\lambda,\mu$ generate a complex
line $L$ in $\mathbb C^g$. Clearly
$\langle \lambda,\mu \rangle \subseteq L \cap \Lambda$.
By Corollary \ref{primitive}, the second property means that 
$\langle \lambda,\mu \rangle$ is maximal. Therefore 
$\langle \lambda,\mu \rangle = L \cap \Lambda$
is the lattice of a complex subtorus.
\end{proof}

Let $E$ be an elliptic curve in $A = \mathbb C^g / \Lambda$. There is
a complex line $L$ in $\mathbb C^g$ such that $L \cap \Lambda$
is a rank 2 submodule of $\Lambda$, and such that $E = L/(L \cap \Lambda)$. 
A basis $\lambda, \mu$ of $L \cap \Lambda$ defines
a bivector $\lambda \wedge \mu$ in $\Lambda \wedge \Lambda$,
determined up to $\pm 1$, hence all positive bases (with respect to the
canonical orientation of $L \cong \mathbb C$) determine
one and the same bivector.

The natural isomorphism $\Lambda \overset{_\sim}{\rightarrow} H_1(A, \mathbb Z)$,
that we write for a while as $\lambda \mapsto \lambda'$,
and the Pontryagin isomorphism $H_1(A, \mathbb Z) \wedge H_1(A, \mathbb Z) 
\overset{_\sim}{\rightarrow} H_2(A, \mathbb Z)$
(see \S \ref{homology}) determine the isomorphism
$$\Lambda \wedge \Lambda \overset{_\sim}{\longrightarrow} H_2(A, \mathbb Z)$$
such that $\lambda \wedge \mu \mapsto \lambda' \star \mu'$.
So, if $E$ is an elliptic curve in $A$, the associated bivector $\lambda \wedge \mu$
corresponds to the fundamental class $\{ E \}$.

\begin{remar} \label{injective} \em
The correspondence sending an elliptic curve $E$ in $A$ to the homology class
$\{E\}$ in $H_2(A,\mathbb Z)$ is injective.
In other words the submodule $\langle \lambda, \mu \rangle$ of $\Lambda$
corresponding to $E$ is reconstructed from the bivector $\lambda \wedge \mu$
in $\Lambda \wedge \Lambda$, which is primitive by Proposition \ref{ellipticsublattice} 
above. If $\lambda \wedge \mu$ is primitive then
$\langle \lambda, \mu \rangle$ is maximal, by Corollary \ref{primitive},
hence coincides with $\mathbb Q \langle \lambda, \mu \rangle \cap \Lambda$
and is uniquely determined by $\lambda \wedge \mu$.
\end{remar}

\section{In the self product of an elliptic curve} \label{Eg}

\subsection{Elliptic curves from endomorphisms} \label{endomorphisms}

Let $E$ be an elliptic curve.
In the self product $E^g$ there are infinitely many elliptic curves.
For every vector $v = (v_1, \ldots, v_g)$ in $\mathbb Z^g$
one has the homomorphism $E \rightarrow E^g$ given by 
$x \mapsto (v_1 x, \ldots, v_g x)$, and if $v$ is primitive
then this is an embedding of $E$ as an elliptic curve $E_v$ in $E^g$.  
We call these the {\em ordinary} elliptic curves in $E^g$. 
They correspond to the bivectors of $\Lambda \wedge \Lambda$ of the form 
$$v \wedge \tau v$$  with $v \in \mathbb Z^g$ primitive. 
In fact, $E_v$ is given by the submodule $\langle v, \tau v \rangle$ of $\Lambda$,
and the basis $v, \tau v$ is positively oriented (since ${\rm im}(\tau) > 0$).

Consider a polarization on $E^g$ which decomposes as a product of pullback polarizations
from the factors, whose theta divisor is of the form
$$p_1^\ast (m_1\Theta) + \cdots + p_g^\ast (m_g\Theta)$$
where $\Theta$ is the theta divisor of the principal polarization on $E$,
and $m_1,\ldots,m_g$ are positive integers.
Then the degree of an ordinary elliptic curve is 
$$\deg (E_v) = m_1v_1^2 + \cdots + m_gv_g^2$$
for one has $E_v \cdot p_i^\ast (\Theta) = \deg \mu_{v_i}^\ast(0) = v_i^2$.

\begin{exe} \label{Eord} \em
We list the ordinary elliptic curves in $E^g$ of degree $\leq g$ 
with respect to the principal polarization (every $m_i = 1$), 
for $g=2$ and $g=3$.  

The ordinary elliptic curves of degree $\leq 2$ in $E^2$ are the two factors, 
with degree = 1, arising from the pairs $(1,0),(0,1)$, the diagonal and 
the anti-diagonal curves, with degree = 2, arising from $(1,1),(1,-1)$.

The ordinary elliptic curves of degree $\leq 3$ in $E^3$ are:
the three coordinate copies of $E$,
with degree = 1, arising from $(1,0,0)$, $(0,1,0)$, $(0,0,1)$;
in each coordinate copy of $E^2$
the diagonal and the anti-diagonal curves, with degree = 2,
arising from $(1, \pm 1, 0)$, $(1,0, \pm 1)$, $(0,1, \pm 1)$;
the four ordinary curves, with degree = 3, arising from $(1, \pm 1, \pm 1)$.
Summing up: 3 + 6 + 4 = 13 ordinary elliptic curves.
\end{exe}

The construction above can be extended using endomorphisms.
Every sequence $(\phi_1, \ldots, \phi_g)$ in ${\rm End}(E)^g$
defines a homomorphism $\phi : E \rightarrow E^g$.
This is an embedding if and only if the only factorizations
$(\phi_1, \ldots, \phi_g) = (\psi_1, \ldots, \psi_g) \circ \sigma$
with every $\psi_i$ and $\sigma$ in ${\rm End}(E)$
are those with $\sigma$ belonging to ${\rm Aut}(E)$.
This is the case, for instance, if some $\phi_i$ is an automorphism.
The image $\phi(E)$ is an elliptic curve in $E^g$, isomorphic to $E$. 
Two sequences $\phi$ and $\psi$ as above determine the same curve 
if and only if there is some automorphism $\sigma$ in ${\rm Aut}(E)$
such that $(\phi_1, \ldots, \phi_g) = (\psi_1, \ldots, \psi_g) \circ \sigma$.

If $\phi$ arises from some $(v_1,\ldots,v_g)$ belonging to $\mathbb Z^g$
then $\phi(E)$ is an ordinary elliptic curve $E_v$. If instead every $\phi \circ \sigma$ 
with $\sigma \in {\rm Aut}(E)$ does not come from  $\mathbb Z^g$ 
then one has an extra-ordinary elliptic curve.
This of course requires that ${\rm End}(E)\neq \mathbb Z$, 
i.e. that $E$ has complex multiplication.

\begin{prop}
If $E$ has no complex multiplication, there are in $E^g$ 
only ordinary elliptic curves.
\end{prop}

\begin{proof}
Any elliptic curve in $E^g$ is isogenous to $E$, and can be obtained as the image
of some homomorphism $\phi : E \rightarrow E^g$, given by a sequence
of endomorphisms of $E$. If $E$ has no complex multiplication then 
$\phi$ arises from some $v$ in $\mathbb Z^g$. Write $v = d v'$ with $v'$ primitive.
Then the image $\phi(E)$ coincides with the ordinary elliptic curve $E_{v'}$.
\end{proof}

Consider on $E^g$ a product polarization, given by integers $m_1,\ldots,m_g$. 
Then the degree of an elliptic curve as above is
$$\deg(\phi(E)) = m_1 \deg(\phi_1) + \cdots + m_g \deg(\phi_g)$$
for the theta divisor on $E^g$ is 
$p_1^\ast(m_1 \Theta) + \cdots + p_g^\ast(m_g \Theta)$ where $\Theta$
is the theta divisor of the principal polarization on $E$, and one has
$\phi(E) \cdot p_i^\ast(\Theta) = \deg(\phi_i)$.

\begin{exe} \label{Eexord} \em
Assume that $E$ is the curve $\tau = i$. 
We list the extra-ordinary elliptic curves in $E^g$ of degree $\leq g$ 
with respect to the principal polarization (every $m_i = 1$), 
for $g=2$ and $g=3$.  

The group ${\rm Aut}(E)$ is identified with the subgroup $\{ \pm 1, \pm i \}$
of $\mathbb C^\times$ and the ring ${\rm End}(E)$ is identified with the ring of 
Gaussian integers $\mathbb Z[i]$ in $\mathbb C$ 
(cf. \cite{AG}, Ch. IV, Ex. 4.20.1). 
There are endomorphisms of degree 2,
namely $\pm(1 \pm i)$, which may be coupled with some automorphism
in the construction described above.

In $E^2$ one has 2 extra-ordinary curves with degree = 2, arising from the pairs
$(1,\pm i)$, the graphs of the automorphisms of $E$ given by $\pm i$.

In $E^3$ one has the following extra-ordinary elliptic curves with degree $\leq 3$:
6 elliptic curves of degree = 2 arising from $(1, \pm i, 0)$, $(1,0, \pm i)$, $(0,1, \pm i)$
(in each coordinate copy of $E^2$ the two extra-ordinary curves described above);
12 elliptic curves of degree = 3 arising from 
$(1, \pm 1, \pm i)$, $(1, \pm i, \pm 1)$, $(1, \pm i, \pm i)$; 
finally, using some endomorphism, one has
12 more elliptic curves of degree = 3 arising from 
$(1, \pm (1 \pm i), 0)$, $(1,0, \pm (1 \pm i))$, $(0,1, \pm (1 \pm i))$.
Summing up: 6 + 12 + 12 = 30 extra-ordinary elliptic curves.
\end{exe}

\subsection{In presence of complex multiplication} \label{CM}

Let $E$ be an elliptic curve, that we identify with 
$\mathbb C / (\mathbb Z + \tau \mathbb Z)$
for some $\tau$ in the upper half plane. 
Assume that $E$ has complex multiplication.
The necessary and sufficient condition is that
$\tau$ is algebraic of degree $2$ over $\mathbb Q$ 
(cf. e.g. \cite{AG}, Ch. IV, Thm. 4.19).
So, assume that $\tau$ satisfies the equation
\begin{equation} \label{tauquad}
\tau^2 + \frac{u}{w} \tau + \frac{v}{w} = 0
\end{equation}
with $u,v,w$ in $\mathbb Z$ such that $w > 0$ and $(u,v,w)=(1)$
and moreover $$u^2 - 4vw < 0$$ as $\tau$ is an imaginary complex number.

The self product $E^g$ is identified with $\mathbb C^g / \Lambda$ where
$\Lambda = \mathbb Z^g + \tau \mathbb Z^g$.
An element of $\Lambda$ is written as $\lambda = \lambda_1 + \tau \lambda_2$
with $\lambda_1, \lambda_2$ in $\mathbb Z^g$.
We exploit in the present setting the conditions of Proposition \ref{ellipticsublattice}.

It follows from (\ref{tauquad}) that multiplication by $w\tau$ in $\mathbb C$
sends the lattice $\mathbb Z + \tau \mathbb Z$ into itself and so defines an
endomorphism of $E$, in other words $w \tau$ may be viewed as an element
of ${\rm End}(E)$. It is well known, by the way, that
${\rm End}_{\mathbb Q}(E)$ is naturally identified with $\mathbb Q(\tau)$.
Similarly, for every $g>1$ multiplication by $w \tau$ in $\mathbb C^g$ induces
an endomorphism of $\Lambda = \mathbb Z^g + \tau \mathbb Z^g$, that we write as 
$\lambda \mapsto \bar\lambda$ where by definition
$$\bar\lambda := w \tau \, \lambda,$$
which defines an endomorphism of $E^g$.

Define $\Lambda_{\mathbb Q} := \Lambda \otimes \mathbb Q$ and
write an identification $\Lambda_{\mathbb Q} \wedge 
\Lambda_{\mathbb Q} = (\Lambda \wedge \Lambda) \otimes \mathbb Q$.

\begin{lem} \label{lambdabarlambda}
If $\omega$ is a primitive decomposable bivector in $\Lambda \wedge \Lambda$,
the following properties are equivalent:
\begin{enumerate}
\item[$-$]
$\omega$ maps to $0$ in $\mathbb C^g \wedge \mathbb C^g$;
\item[$-$]
$\omega = k (\lambda \wedge \bar \lambda)$ in 
$\Lambda_{\mathbb Q} \wedge \Lambda_{\mathbb Q}$
for some $\lambda \in \Lambda_{\mathbb Q}$ and $k \in \mathbb Q$;
\item[$-$]
$\omega = \pm (1/r) (\lambda \wedge \bar \lambda)$ in $\Lambda \wedge \Lambda$
where $\lambda \in \Lambda$ is primitive and
$r$ is the integral content of $\lambda \wedge \bar \lambda$.
\end{enumerate}
Firthermore, every bivector of the form $$(1/r) (\lambda \wedge \bar \lambda)$$
is primitive and decomposable in $\Lambda \wedge \Lambda$
and is canonically oriented.
\end{lem}

\begin{proof}
$1 \Rightarrow 2$. Write $\omega = \lambda \wedge \mu$. 
If $\omega$ maps to $0$ then there is $z \in \mathbb C$ such that $z \lambda = \mu$.
This requires that $z \in \mathbb Q(\tau)$ and, in the present setting,
we can write $z = x + \tau y$ with $x,y$ in $\mathbb Q$. Hence
$\mu = x \lambda + y (\tau\lambda) = x \lambda + (y/w) \bar\lambda$
in $\Lambda_{\mathbb Q}$, and $\omega = (y/w) (\lambda \wedge \bar\lambda)$.
$2 \Rightarrow 3$.  If $\omega = k (\lambda \wedge \bar \lambda)$,
we may assume that $\lambda$ belongs to $\Lambda$ and is primitive;
write $k = q/r$ with $q,r$ coprime integers and $r>0$;
then $r\, \omega = q (\lambda \wedge \bar\lambda)$
with $\omega$ primitive implies that
$q = \pm 1$, and that $r$ is the integral content of $\lambda \wedge \bar\lambda$.
That $3 \Rightarrow 1$ is clear.

Let now $\lambda$ be a primitive element of $\Lambda$, and let
$r$ be the content of $\lambda \wedge \bar\lambda$.
Then $r$ divides $[\bar\lambda]$ in $\Lambda / \mathbb Z \lambda$,
by Lemma \ref{content}, thus $[\bar\lambda] = r [\mu]$
for some $\mu$, and $(1/r) (\lambda \wedge \bar\lambda)
= \lambda \wedge \mu$ is decomposable in $\Lambda \wedge \Lambda$.
The pair $\lambda,\mu$ is equally oriented as the pair
$\lambda,\bar\lambda$ and the pair $\lambda,\tau\lambda$, which 
is equally oriented as the pair $\lambda,i\lambda$, since ${\rm im}(\tau) > 0$.
\end{proof}

\begin{prop}
If $E$ has complex multiplication, for every $g>1$ there are in $E^g$ 
infinitely many extra-ordinary elliptic curves.
\end{prop}

\begin{proof}
If $\lambda \wedge \mu$ is of the special form above, it is easy to see that
$\langle \lambda, \mu \rangle \cap \mathbb Z^g \neq 0$ if and only if 
$\langle \lambda, \bar\lambda \rangle \cap \mathbb Z^g \neq 0$.
Remark that if $\lambda = \lambda_1 + \tau \lambda_2$ then
$\bar \lambda = -v \lambda_2 +  \tau (w \lambda_1 - u \lambda_2)$.
It is easy to see that $\langle \lambda, \bar\lambda \rangle \cap \mathbb Z^g \neq 0$
happens if and only if $\lambda_1, \lambda_2$ are linearly dependent.
Since $\langle \lambda, \mu \rangle \cap \mathbb Z^g \neq 0$ for an ordinary
bivector $v \wedge \tau v$, it is enough to take $\lambda_1, \lambda_2$ to be linearly
independent in order to have an extra-ordinary bivector $\lambda \wedge \mu$.
\end{proof}

If $\omega = (1/r) (\lambda \wedge \bar \lambda)$ in $\Lambda \wedge \Lambda$
is of the special form in Lemma \ref{lambdabarlambda} above,
we denote by $E_\omega$ the elliptic curve in $E^g$
whose associated bivector is $\omega$.

Assume now that a product polarization is given on $E^g$.
Write a decomposition $\omega = \lambda \wedge \mu$, as in the proof of the Lemma.
Then $(\lambda, \mu) = (1/r) (\lambda, \bar \lambda)$, 
where $r$ divides $(\lambda, \bar \lambda)$ because 
it divides  $\lambda \wedge \bar \lambda$. 
The basis $\lambda, \mu$ is positive, by the Lemma,
hence is not symplectic (see \S \ref{homology}). Therefore (see \S \ref{degpol}) 
the degree of the corresponding elliptic curve is given by $- (\lambda,\mu)$, and
\begin{equation} \label{degree}
\deg (E_\omega) = -(1/r) (\lambda, \bar \lambda).
\end{equation}

\section{Diophantine equations for the collection 
of elliptic curves in the self product}

Let $E$ be an elliptic curve with complex multiplication.
Still we refer to the setting of \S \ref{CM}.
We characterize the bivectors of the special form occurring in 
Lemma \ref{lambdabarlambda} by means of coordinates and
explicit equations in the coordinates, for small values of the dimension $g$.
Given a primitive bivector $\omega$ in $\Lambda \wedge \Lambda$,
we have to understand when there are 
$\lambda \in \Lambda_{\mathbb Q}$ and $k \in \mathbb Q$ such that
$k\, \lambda \wedge \bar\lambda = \omega$
holds in $\Lambda_{\mathbb Q} \wedge \Lambda_{\mathbb Q}$. 

\subsection{In dimension $g = 2$} \label{g=2}

In $\mathbb C^2$ is the period lattice $\Lambda \cong \mathbb Z^4$
of the Abelian surface $E^2$. 
If $\lambda = (a + \tau c,\, b + \tau d)$
is the element of $\Lambda $ of coordinates $(a,b,c,d)$,
then $\bar \lambda $ has coordinates $(-vc, -vd, wa-uc, wb-ud)$.

The bivector $\lambda \wedge \bar \lambda$
in $\Lambda \wedge \Lambda \cong \mathbb Z^6$  
is given by the following 6 coordinates: the entries of the matrix
\begin{equation*}
\left(\begin{matrix}
wa^2-uac+vc^2 & wab-uad+vcd \\ wab-ubc+vcd & wb^2-ubd+vd^2
\end{matrix}\right) \ = \
\left(\begin{matrix} a & c \\ b & d \end{matrix}\right)
\left(\begin{matrix} w & -u \\ 0 & v \end{matrix}\right)
\left(\begin{matrix} a & b \\ c & d \end{matrix}\right)
\end{equation*}
and the entries of the vector  $(v, w)\, (ad-bc).$
Note that the difference of entries in the second diagonal is $u(ad-bc)$.
As $u,v,w$ are coprime, there are only 4 essential coordinates, namely
the upper triangular entries of the matrix above and the scalar $$ad-bc,$$
and the content $r$ of $\lambda \wedge \bar \lambda$ is the greatest common divisor
of the four essential coordinates.

Because of Lemma \ref{lambdabarlambda}, the bivectors of the form
$\pm (1/r)(\lambda \wedge \bar \lambda)$ are in one to one correspondence
with the primitive vectors $(\alpha , \beta, \gamma, \eta)$ in $\mathbb Z^4$
such that there are $(a,b,c,d) \in \mathbb Q^4$ and $k \in \mathbb Q$ with
\begin{equation} \label{ADA2}
k \left(\begin{matrix} a & c \\ b & d \end{matrix}\right)
\left(\begin{matrix} w & -u \\ 0 & v \end{matrix}\right)
\left(\begin{matrix} a & b \\ c & d \end{matrix}\right) =
\left(\begin{matrix} \alpha  & \gamma \\ 
\gamma + u\eta  & \beta \end{matrix}\right),
\end{equation} 
\begin{equation} \label{det2}
k\, (ad-bc) = \eta,
\end{equation}
so the bivectors $(1/r)(\lambda \wedge \bar \lambda)$ are characterized by
the condition that $\alpha , \beta \geq 0$.

Taking the determinant in (\ref{ADA2}) and using (\ref{det2}), it follows that
\begin{equation} \label{alfabeta}
\alpha \beta - \gamma (\gamma + u \eta) = vw \eta ^2.
\end{equation}

\begin{thm} \label{g=2}
The bivectors of the form $(1/r)(\lambda \wedge \bar \lambda)$ 
in $\Lambda \wedge \Lambda$ are in one to one correspondence
with the primitive vectors $(\alpha, \beta, \gamma, \eta)$ in $\mathbb Z^4$
with $\alpha,\beta \geq 0$ satisfying condition $(\ref{alfabeta})$.
\end{thm}

\begin{proof}
Let $(\alpha, \beta, \gamma, \eta)$ be a primitive vector,
and assume that condition $(\ref{alfabeta})$ holds. We show that equations 
(\ref{ADA2}) and (\ref{det2}) admit some solution.

Multiplying (\ref{ADA2}) by the adjoint matrix
and using (\ref{det2}), it follows that
\begin{equation} \label{linear} 
\eta \left(\begin{matrix} a & c \\ b & d \end{matrix}\right)
\left(\begin{matrix} w & -u \\ 0 & v \end{matrix}\right) =
\left(\begin{matrix} \alpha  & \gamma \\ 
\gamma + u\eta & \beta \end{matrix}\right)
\left(\begin{matrix} d & -b \\ -c & a \end{matrix}\right).
\end{equation}
From this in turn, by multiplication again, it follows that
$$\eta \left(\begin{matrix} a & c \\ b & d \end{matrix}\right)
\left(\begin{matrix} w & -u \\ 0 & v \end{matrix}\right)
\left(\begin{matrix} a & b \\ c & d \end{matrix}\right) =
\left(\begin{matrix} \alpha  & \gamma  \\ 
\gamma + u \eta  & \beta  \end{matrix}\right)  (ad-bc).$$

We claim that the vectors $(a,b,c,d)$ which solve equations 
(\ref{ADA2}) and (\ref{det2}) with some $k$ 
are the nonzero solutions of equation (\ref{linear}).
If $\eta \neq 0$, in the last equation above for every nonzero solution one has 
$ad-bc \neq 0$, so that (\ref{ADA2}) and (\ref{det2}) are satisfied with 
$k = \eta/(ad-bc)$. 
If $\eta = 0$, in equation $(\ref{linear})$ 
for every nonzero solution one has $ad-bc=0$; writing
$\left(\begin{matrix} a & b \\ c & d \end{matrix}\right) = 
\left(\begin{matrix} p \\ q \end{matrix}\right)
\left(\begin{matrix} r & s \end{matrix}\right)$  it follows that
$\left(\begin{matrix} \alpha & \gamma \\ \gamma & \beta \end{matrix}\right) =
t \left(\begin{matrix} r \\ s \end{matrix}\right)
\left(\begin{matrix} r & s \end{matrix}\right)$;
it is then easy to check that (\ref{ADA2}) and (\ref{det2}) are satisfied
with $k = t/(wp^2-upq+vq^2)$. 

Therefore equations $(\ref{ADA2})$ and $(\ref{det2})$ 
have a solution if and only if equation $(\ref{linear})$ has a nonzero solution.
Write it in the form
\begin{equation*} \tag{\ref{linear}$'$} \label{linearb}
\left(\!\begin{matrix}
w \eta  & 0 & \gamma  & -\alpha  \\
-\gamma - u \eta & \alpha  & v \eta  & 0 \\
0 & w \eta  & \beta  & -\gamma - u \eta \\
-\beta  & \gamma & 0 & v \eta
\end{matrix}\;\right)
\left(\begin{matrix} a \\ b \\ c \\ d \end{matrix}\right) = 
\left(\begin{matrix} 0 \\ 0 \\ 0 \\ 0 \end{matrix}\right).
\end{equation*}
It is not difficult to check that the determinant of the matrix above is equal to 
$-(\alpha\beta - \gamma(\gamma + u\eta) - vw \eta^2)^2$, that under condition
(\ref{alfabeta}) reduces to $0$.
\end{proof}

\begin{remar} \label{rank=2} \em
We have seen in the proof above that, under condition (\ref{alfabeta}), 
the vectors which solve equations (\ref{ADA2}) and (\ref{det2})
with some $k$ are the nonzero solutions of equation (\ref{linearb}).
It is not difficult to check that the coefficient matrix above has rank $=2$ 
and moreover, if $\alpha \beta \neq 0$, the first two rows and the last two rows 
are independent.
\end{remar}

Assume that on $E^2$ is given a product polarization, corresponding to the pair 
of positive integers $(m,n)$. Then one has 
$- (\lambda, \bar\lambda) = m(wa^2 - uac+ vc^2) + n(wb^2 - ubd + vd^2)$, 
hence formula (\ref{degree}) for the degree of the corresponding elliptic curve 
may be written as
\begin{equation}
\deg (E_\omega) = m\alpha  + n\beta
\end{equation}
expressed in terms of the corresponding coordinates in $\mathbb Z^4$.

\begin{remar} \em
It is not difficult to show, working with the explicit equations in Theorem \ref{g=2}, 
that the maximum value of the quantity $N_{E^2}(2)$ is $8$, and is attained
only with respect to the principal polarization ($m=n=1$).
\end{remar}

\subsection{In dimension $g = 3$} \label{g=3}

In $\mathbb C^3$ one has the period lattice $\Lambda \cong \mathbb Z^6$
of the Abelian threefold $E^3$. If 
$\lambda  = (a + \tau d, b + \tau e, c + \tau f)$
is the element of $\Lambda$ of coordinates $(a,b,c,d,e,f)$
then $\bar \lambda = (-vd, -ve, -vf, wa-ud, wb-ue, wc-uf)$.

The bivector $\lambda \wedge \bar \lambda$
in $\Lambda \wedge \Lambda \cong \mathbb Z^{15}$ 
is given by the following 15 coordinates: the entries of the matrix
\begin{equation*}
\left(\begin{matrix} a & d \\ b & e \\ c & f \end{matrix}\right)
\left(\begin{matrix} w & -u \\ 0 & v \end{matrix}\right)
\left(\begin{matrix} a & b & c \\ d & e & f \end{matrix}\right)
\end{equation*}
and those in the matrix
$\left(\begin{matrix} v \\ w \end{matrix}\right) 
\Big(\begin{matrix} ae-bd & af-cd & bf-ce \end{matrix}\Big).$
Note that the differences of two entries in symmetric positions 
in the $3 \times 3$  matrix are the entries of 
$u\, \Big(\begin{matrix} ae-bd & af-cd & bf-ce \end{matrix}\Big)$.
As $u,v,w$ are coprime, there are only 9 essential coordinates, the upper
triangular entries of the matrix above and the entries of the vector
\begin{equation*}
\Big(\begin{matrix} ae-bd & af-cd & bf-ce \end{matrix}\Big),
\end{equation*}
and the content $r$
of $\lambda \wedge \bar \lambda$ is the greatest common divisor
of the nine essential coordinates.

Because of Lemma \ref{lambdabarlambda}, the bivectors of the form
$\pm (1/r)(\lambda \wedge \bar \lambda)$ are in one to one correspondence
with the primitive vectors
$(\alpha, \beta, \gamma, \delta, \epsilon, \zeta, \eta, \theta, \iota)$
in $\mathbb Z^9$ such that there are $(a,b,c,d,e,f) \in \mathbb Q^6$ 
and $k \in \mathbb Q$ with
\begin{equation} \label{ADA3}
k \left(\begin{matrix} a & d \\ b & e \\ c & f \end{matrix}\right)
\left(\begin{matrix} w & -u \\ 0 & v \end{matrix}\right)
\left(\begin{matrix} a & b & c \\ d & e & f \end{matrix}\right) =
\left(\begin{matrix} 
\alpha  & \delta  & \epsilon  \\ \delta + u \eta  & \beta  & \zeta  \\ 
\epsilon + u \theta  & \zeta + u \iota  & \gamma  
\end{matrix}\right),
\end{equation} 
\begin{equation} \label{det3}
k\; \big(ae-bd,\, af-cd,\, bf-ce \big) 
= \big(\eta , \theta , \iota \big),
\end{equation}
and the bivectors $(1/r)(\lambda \wedge \bar \lambda)$
are characterized  requiring that $\alpha, \beta, \gamma \geq 0$.

The system of quadratic equations (\ref{ADA3}) and (\ref{det3})
can be split into three systems like (\ref{ADA2}) and (\ref{det2}),
corresponding to the principal minors 
$$\left(\begin{matrix} 
\alpha & \delta \\ \delta + u \eta & \beta
\end{matrix}\right)  \hspace{20pt}
\left(\begin{matrix} 
\beta & \zeta \\ \zeta + u \iota & \gamma  
\end{matrix}\right)  \hspace{20pt}
\left(\begin{matrix} 
\gamma  & \epsilon + u \theta \\ \epsilon & \alpha 
\end{matrix}\right).$$
It follows from \S \ref{g=2} that
we have the following necessary conditions
\begin{equation} \label{alfabetagamma} 
\alpha  \beta  - \delta (\delta + u \eta) = vw \eta ^2  \hspace{10pt}
\beta \gamma  - \zeta (\zeta + u \iota) = vw \iota ^2  \hspace{10pt}
\alpha \gamma  - \epsilon (\epsilon + u \theta) = vw \theta ^2.
\end{equation}

As in the preceding section, we will replace the quadratic equations
(\ref{ADA3}) and (\ref{det3}) with a system of linear equations, 
but now another nontrivial 
compatibility condition is necessary, given by the equation
\begin{equation} \label{cubic} \begin{split}
2\alpha \beta \gamma \  = & \ \
\delta (\epsilon + u \theta) \zeta  + 
(\delta + u \eta) \epsilon (\zeta + u \iota) + \\
& \ \ vw \big((2\delta + u \eta) \theta \iota -
(2\epsilon + u \theta) \eta \iota +
(2\zeta + u \iota) \eta \theta \big). \end{split}
\end{equation}

\begin{thm} \label{g=3}
The bivectors of the form $(1/r)(\lambda \wedge \bar \lambda)$ 
in $\Lambda \wedge \Lambda$ correspond bijectively to the primitive vectors 
$(\alpha, \beta, \gamma, \delta, \epsilon, \zeta,\eta, \theta, \iota)$
in $\mathbb Z^9$ with $\alpha, \beta, \gamma \geq 0$
satisfying conditions $(\ref{alfabetagamma})$ and $(\ref{cubic})$.
\end{thm}

\begin{proof}
Given a primitive vector in $\mathbb Z^9$, assume that equations 
(\ref{ADA3}) and (\ref{det3}) have a solution. We know that
(\ref{alfabetagamma}) is satisfied, and we will show that (\ref{cubic}) is too.
If $\alpha  =0$ then clearly $a=d=0$, hence
$\delta = \epsilon = 0$ and $\eta  = \theta  = 0$, and
(\ref{cubic}) is satisfied. The same if $\beta  =0$ or $\gamma  =0$. 
So assume that $\alpha  \beta \gamma  \neq 0$.

The system of quadratic equations (\ref{ADA3}) and (\ref{det3})
can be split into three systems like (\ref{ADA2}) and (\ref{det2}),
corresponding to the principal minors in the right hand side of (\ref{ADA3}).
According to Remark \ref{rank=2}, such a quadratic system 
can be replaced with a linear equation like (\ref{linearb}),
whose nonzero solutions are the solutions of the quadratic system,
and since $\alpha\beta\gamma \neq 0$ we know moreover that 
for each matrix equation like (\ref{linearb}) we only need to consider
the first two scalar components. So we end with the linear equation
\begin{equation} \label{abcdef}
\left(\begin{matrix}
w \eta & 0 & 0 & \delta & -\alpha & 0 \\
- \delta - u \eta & \alpha  & 0 & v \eta & 0 & 0 \\
0 & w \iota & 0 & 0 & \zeta & -\beta \\
0 & - \zeta - u \iota  & \beta & 0 & v \iota & 0 \\
-\gamma & 0 & \epsilon & 0 & 0 & v \theta \\
0 & 0 & w \theta & \gamma & 0 & - \epsilon - u \theta \\
\end{matrix}\right) 
\left(\begin{matrix} a \\ b \\ c \\ d \\ e \\ f \end{matrix}\right) =
\left(\begin{matrix} 0 \\ 0 \\ 0 \\ 0 \\ 0 \\ 0 \end{matrix}\right)
\end{equation}
having a nonzero solution. The determinant of the coefficient matrix must be zero.
It is easy to see that, under condition (\ref{alfabetagamma}),
the determinant of the matrix above is given by $\alpha \beta \gamma$
multiplying the expression 
\begin{equation*} \begin{split}
2\alpha \beta \gamma \ \ - & \ \
\delta (\epsilon + u \theta) \zeta  - 
(\delta + u \eta) \epsilon (\zeta + u \iota) \\
- & \ \ vw \big((2\delta + u \eta) \theta \iota -
(2\epsilon + u \theta) \eta \iota +
(2\zeta + u \iota) \eta \theta \big). \end{split}
\end{equation*}
Thus condition (\ref{cubic}) is satisfied.

Conversely, given a primitive vector in $\mathbb Z^9$, assume that 
conditions (\ref{alfabetagamma}) and (\ref{cubic}) are satisfied.
We will show that equations (\ref{ADA3}) and (\ref{det3}) have a solution. 
If $\alpha = 0$ then clearly $\delta = \eta = 0$ and $\epsilon = \theta = 0$
follow from (\ref{alfabetagamma}) and in (\ref{ADA3}) we must have $a=d=0$.
So (\ref{ADA3}) and (\ref{det3}) reduce to equations like
(\ref{ADA2}) and (\ref{det2}) involving $(b,c,e,f)$, whose compatibility condition
is an equation like (\ref{alfabeta}), according to Theorem \ref{g=2},
and is satisfied because of (\ref{alfabetagamma}). Hence there is a 
solution of equations (\ref{ADA3}) and (\ref{det3}). The same if
$\beta = 0$ or $\gamma = 0$. So assume that $\alpha \beta \gamma \neq 0$.

Let us show that every nonzero solution $(a,b,c,d,e,f)$ of equation (\ref{abcdef}) above
gives a solution of equations (\ref{ADA3}) and (\ref{det3}). View it as
a $3 \times 2$ coordinate matrix. All its $2 \times 2$ minors are nonzero. 
If for instance $(a,c,d,f)$ were zero with $(b,e)$ nonzero,
equation (\ref{abcdef}) would give $\alpha = 0$, contrary to the hypothesis.
Each of these minors satisfies an
equation like (\ref{linearb}) and therefore satisfies equations like (\ref{ADA2})
and (\ref{det2}) for some proportionality factor, according to
Remark \ref{rank=2}. Since $\alpha\beta\gamma \neq 0$ the three proportionality
factors are equal. Because two distinct equations like (\ref{ADA2}) 
have in common one scalar equation which determines the factor uniquely.
Therefore we have a solution of equations (\ref{ADA3}) and (\ref{det3}).
\end{proof}

If on $E^3$ a product polarization is given, corresponding to the sequence 
of positive integers $(m,n,p)$, then one has 
$- (\lambda, \bar\lambda) =
m(wa^2 - uad + vd^2) + n(wb^2 - ube + ve^2) + p(wc^2 - ucf + vf^2)$, 
hence by formula (\ref{degree}) the degree of the elliptic curve is 
\begin{equation}
\deg (E_\omega) = m\alpha  + n\beta + p\gamma
\end{equation}
expressed in terms of the corresponding coordinates in $\mathbb Z^9$.

\begin{remar} \em
It is not difficult to show, working with the equations in Theorem \ref{g=3}, 
that the maximum value of the quantity $N_{E^3}(3)$ is $43$, and clearly is attained
only with respect to the principal polarization ($m=n=p=1$). 
In Example \ref{Eexord} the maximum is attained,
and all elliptic curves are obtained with a natural construction.
\end{remar}

\section{On the number of solutions}

\subsection{Some results from Number Theory} \label{numbertheory}

We collect some known results concerning the following problem: 
given a compact convex subset $K$ in $\mathbb R^g$, estimate the number 
$N := {\rm card}\, (\mathbb Z^g \cap K)$ of integer vectors (or lattice points) 
belonging to the convex set. This number is naturally approximated by 
the $g$-dimensional volume $V:= {\rm vol}(K)$, and then the question is 
to estimate the (error or) discrepancy $N-V$.
Let us confine ourselves to small dimensions, and to upper estimates. One has:
\begin{itemize}
\item[$(a)$]
for a compact convex region in $\mathbb R^2$ of area $A$
whose boundary is a Jordan curve of length $L$ then 
$N \leq A + \frac{1}{2} L + 1$ (cf. Nosarzewska \cite{N});
\item[$(b)$]
for a compact convex region in $\mathbb R^3$ of volume $V$
whose boundary is a differential surface of class $C^2$
with surface area $S$ and total mean curvature $M$ then 
$N \leq V + (1/2) S + (1/\pi) M + 1$  (cf. Overhagen \cite{O}).
\end{itemize}

We apply these results in the following way.
Given a compact convex subset $K$ in $\mathbb R^g$,
for every scale factor $t \in \mathbb R_{\geq 0}$ denote by $N(t)$
the number of lattice points in the deformed region $\sqrt t \, K$. Then:
\begin{itemize}
\item[$(a')$] \label{nosarzewska}
in the setting of $(a)$ one has
$N(t) \leq A\, t + \frac{L}{2}\, t^{1/2} + 1$;
\item[$(b')$] \label{overhagen}
in the setting of $(b)$ one has
$N(t) \leq V\, t^{3/2} + \frac{S}{2}\, t + \frac{M}{\pi}\, t^{1/2} + 1$.
\end{itemize}

In both cases the leading term of the bounding function is $V\, t^{g/2}$
where $V$ is the $g$-dimensional volume.
We emphasize that the inequalities above are valid for arbitrary $t$. 
But if one is content with asymptotic estimates
in the form $N(t) = V\, t^{g/2} + O(t^e)$, holding for $t \gg 0$,
the exponent $e$ may be lowered, and precisely:
\begin{itemize}
\item[$-$]
in the setting of $(a)$ one can take
$e = \frac{131}{416} + \epsilon$ (cf. Huxley \cite{Hu});
\item[$-$]
in the setting of $(b)$, in the special case of a rational ellipsoid, one can take
$e = \frac{21}{32} + \epsilon$ (cf. Chamizo \cite{C1}).
\end{itemize}

\subsection{Bounds for the counting function}

\begin{prop} \label{Ndim2}
For $g=2$ there is an upper bound of the form
$$N_{E^2}(t) \leq C\, t^3 + q(t)$$
where $C$ is a constant and $q$ is a function of order $O(t^{2})$,
and both can be given explicitely.
\end{prop}

\begin{proof}
If $E$ is without complex multiplication, there are only ordinary elliptic curves in $E^2$,
and those with degree $\leq t$ correspond bijectively to the primitive vectors $(a,b)$ 
such that $ma^2+nb^2 \leq t$, taken up to $\pm 1$, whose number has 
an upper bound of the form $A_0t + (1/2)L_0 t^{1/2} + 1$ 
(by $(a')$ in \S \ref{numbertheory}).

Assume that $E$ has complex multiplication. 
The elliptic curves in $E^2$ with degree $\leq t$ correspond bijectively to
the primitive vectors $(\alpha, \beta, \gamma, \eta)$ in $\mathbb Z^4$
with $\alpha,\beta \geq 0$ such that 
$\alpha \beta = \gamma^2 + u \gamma \eta + vw \eta^2$ and
$m \alpha + n \beta \leq t$ (see \S \ref{g=2}).
Define $Q(\gamma,\eta) := \gamma^2 + u \gamma \eta + vw \eta^2$,
a positive definite quadratic form.
The integer $\min \{m,n\}$ is important in the following computations.
We may assume that $m \leq n$.
Then write $t = m t' + r$ with $r<m$ so that $t' =[t/m]$.

The vectors in $\mathbb Z^4$ satisfying the conditions above correspond,
forgetting $\beta$, to (possibly non-primitive) vectors $(\alpha,\gamma,\eta)$ 
in $\mathbb Z^3$ such that $0 \leq \alpha \leq t'$ and
$$Q(\gamma,\eta) \leq \alpha (t' - \alpha).$$
For $\alpha = 0$ the primitive vector $(0,1,0,0)$ corresponds to $(0,0,0)$,
while for $\alpha \geq 1$ the element $\beta = (1/\alpha) Q(\gamma,\eta)$ 
is determined by $(\alpha,\gamma,\eta)$. 
Thus the correspondence above is injective
and we have $$N_{E^2}(t) \leq R(t)$$
where $R(t)$ is the number of solutions in $\mathbb Z^3$ 
of the inequalities above. Furthermore
$$R(t) = \sum_{\alpha=0}^{t'} R(\alpha,t)$$ where $R(\alpha,t)$ 
is the number of vectors $(\gamma,\eta)$ in $\mathbb Z^2$ such that
$Q(\gamma,\eta) \leq \alpha (t' - \alpha)$.
By $(a')$ in \S \ref{numbertheory} we have
$$R(\alpha,t) \leq A\, \alpha (t' - \alpha) + 
\frac{L}{2}\, \big(\alpha (t' - \alpha) \big)^{1/2}+1$$ 
where $A = \pi (vw - (1/4)u^2)^{-1/2}$ is the area of the region 
$Q(\gamma,\eta) \leq 1$ in $\mathbb R^2$
and $L$ is the length of the bounding ellipse.

We have to estimate the sum over $\alpha$ of terms in the formula above.
For the first summation we have an exact formula
$$\sum_{0}^{t'} \alpha (t' - \alpha) = 
\frac{1}{6} t'(t'+1)(t'-1).$$
For the second summation we use a basic approximation method,
explained in Remark \ref{trapezoidal} below,
and we find the following upper bound
$$\sum_{0}^{t'} [\alpha (t' - \alpha)]^{1/2} \leq 
\frac{\pi}{8}\, {t'}^2 - \frac{t'-2}{6t'}.$$
Summing up it follows that
$$R(t) \leq  A \left(\frac{1}{6} {t'}^3 - \frac{1}{6} t' \right) + 
\frac{L}{2} \left(\frac{\pi}{8} {t'}^2 - \frac{t'-2}{6t'} \right) +  t',$$
and this in turn is bounded above using $t' \leq t/m$.
\end{proof}

\begin{remar} \label{trapezoidal} \em
In the interval $[0,t]$, with $t$ a positive integer, for the function $f(x) := {x(t-x)}^{1/2}$, 
applying the approximation method known as the 
'trapezoidal rule', in the interval $[1,t-1]$ and for $t \geq 2$, we have that 
$\int_{1}^{t-1} f(x)dx - \sum_{n=1}^{t-1} f(n) = - \frac{t-2}{12} f''(\xi)$
for some $\xi$ in $[1,t-1]$; since for $f''$ the maximum value is $f''(t/2) = - 2/t$, 
and since $\int_{0}^{t} f(x)dx = \frac{\pi}{8}t^2$, it follows that
$\sum_{n=0}^{t} f(n) \leq \frac{\pi}{8}\, t^2 - \frac{t-2}{6t}$
that clearly also holds for $t=1$.
\end{remar}

\begin{prop} \label{Ndim3}
For $g=3$ there is an upper bound of the form
$$N_{E^3}(t) \leq C t^5 + q(t)$$ 
where $C$ is a constant and $q$ is a function of order $O(t^{4})$,
and both can be given explicitely.
\end{prop}

\begin{proof}
If $E$ is without complex multiplication, there are only ordinary elliptic curves in $E^3$,
and those with degree $\leq t$ are as many as the primitive vectors $(a,b,c)$ 
such that $ma^2+nb^2+pc^2 \leq t$, taken up to $\pm 1$, 
whose number has an upper bound of the form
$C_0t^{3/2}+q_0(t)$ (by $(b')$ in \S \ref{numbertheory}).

Assume that $E$ has complex multiplication. The elliptic curves in $E^3$ 
with degree $\leq t$ correspond bijectively to the primitive vectors
$(\alpha , \beta , \gamma , \delta , \epsilon , \zeta , \eta , \theta , \iota )$
in $\mathbb Z^9$ with $\alpha,\beta,\gamma \geq 0$ satisfying
conditions (\ref{alfabetagamma}) and (\ref{cubic}) and such that
$m \alpha + n \beta + p \gamma \leq t$ (see \S \ref{g=3}).
We may assume that $m \leq n \leq p$.

Choose $\alpha,\beta,\delta,\eta$ (possibly non-primitive) satisfying 
$\alpha\beta = Q(\delta,\eta)$ with $m \alpha + n \beta \leq t$.
The same argument used for Proposition \ref{Ndim2} above, with a slight 
modification, shows that the number of possibile choices is bounded above by
$$A \left(\frac{1}{6} {t'}^3 - \frac{1}{6} t' \right) + 
\frac{L}{2} \left(\frac{\pi}{8} {t'}^2 - \frac{t'-2}{6t'} \right) +  2t'.$$
The difference in the last summand is because for $\alpha = 1$ we are now
allowing all vectors $(0,\beta,0,0)$ with $0 \leq \beta \leq t/n$,
whose number is $[t/n]+1 \leq t'+1$, and so we add $t'$ to the quantity $R(t)$.

Then choose $\zeta,\iota$ such that $\beta$ divides $Q(\zeta,\iota)$, so 
$\gamma$ is determined from $\beta\gamma = Q(\zeta,\iota)$, and such that
$Q(\zeta,\iota) \leq \beta (1/p) (t - m\alpha - n\beta)$
which implies $Q(\zeta,\iota) \leq [t/n] [t/p]$. Define $r := [t/n] [t/p]$.
The number of possible choices is bounded above by
$$Ar + \frac{L}{2} {r}^{1/2} + 1$$
where $A$ and $L$ are the same as above
(by $(a')$ in \S \ref{numbertheory}).

The remaining unknowns $\epsilon,\theta$ have to satisfy the third
equation in (\ref{alfabetagamma}) and the equation given by (\ref{cubic}),
a quadratic and a linear equation in $\epsilon,\theta$, with at most two common solutions.
The total number of vectors in $\mathbb Z^9$ satisfying the
conditions of the Proposition is therefore bounded above by 
twice the product of the two expressions given above,
which in turn can be bounded above by
$C t^5 + q(t)$
where $C = \frac{1}{3} \frac{A^2}{m^3np}$
and $q$ is a function of order $O(t^4)$ which can be made explicit.
\end{proof}

The estimates given in the previous statements are valid for arbitrary $t$.
We remark that, using the same arguments, it is possible to give asymptotic estimates 
in the form $N_{E^g}(t) = Ct^n + O(t^e)$ with the same leading exponent $n$
and with a smaller discrepancy order $e$, relying on the asymptotic results
mentioned at the end of \S \ref{numbertheory}.

\section{Some results for Abelian varieties of small \newline dimension}
\label{Ag=23}

The previous results for self products of elliptic curves
imply analogous results for polarized Abelian varieties
of small dimensions $g=2$ and $g=3$.
We make use of the Poincar\'e reducibility theorem, in the
following form. If $A$ is a polarized Abelian variety and $B$
is an Abelian subvariety of $A$, there is a unique Abelian subvariety
$B'$ of $A$ such that the sum homomorphism $B \times B' \rightarrow A$
is an isogeny and the pullback polarization on $B \times B'$ is the product
of the pullback polarizations from $B$ and $B'$
(cf. \cite{CAV}, Theorem 5.3.5 and Corollary 5.3.6).

Let $A$ be a polarized Abelian surface. The following are basic facts,
following from the reducibility theorem.
If $A$ contains an elliptic curve $E$ then it also contains a
complementary elliptic curve $E'$ and there is an isogeny 
$E \times E' \rightarrow A$, where the induced polarization 
on $E \times E'$ is a product polarization;
if $A$ contains one more elliptic curve then $E$ and $E'$ are isogenous,
so there is an isogeny $E^2 \rightarrow A$, and the induced polarization
on $E^2$ is a product polarization again;
if there is an isogeny $E^2 \rightarrow A$ then 
there are in $A$ infinitely many elliptic curves, the images of
the ordinary elliptic curves in $E^2$;
in this case, more precisely, we have seen that there are in $A$ 
(infinitely many) extra-ordinary elliptic curves 
if and only if $E$ has complex multiplication.

Furthermore, the function $N_{A}(t)$ can be given an
upper bound which is asymptotically of order $O(t^3)$.
(The estimate for $E^2$ with a product polarization is in Proposition \ref{Ndim2}; 
if there is an isogeny $E^2 \rightarrow A$, of degree $d$, preserving
the polarizations, then $N_{A}(t) \leq N_{E^2}(dt)$ (see \S \ref{isogenies});
if $A$ contains at least three elliptic curves then there is an isogeny as above, 
and the statement follows.)

Let $A$ be a polarized Abelian threefold. The following are basic facts.
If $A$ contains two elliptic curves $E,E'$ then it contains another
elliptic curve $E''$ and there is an isogeny 
$E \times E' \times E'' \rightarrow A$ and,
because of the reducibility theorem, we may assume that
the pullback polarization on $E \times E' \times E''$ is a product polarization;
if $A$ contains one more elliptic curve then at least two of $E,E',E''$ are isogenous
and there are infinitely many ordinary  elliptic curves in $A$.
What we add to the information above, confining ourselves to a rather
coarse analysis for simplicity, is the following.
If there is in $A$ some extra-ordinary elliptic curve then
there are infinitely many extra-ordinary elliptic curves.

Furthermore, the function $N_A(t)$ can be given
an upper bound which is asymptotically of order $O(t^5)$.
(The estimate for $E^3$ is in Proposition \ref{Ndim3},
for $E^2 \times E'$ with $E,E'$ not isogenous follows from the results on surfaces;
if there is an isogeny $E^2 \times E' \rightarrow A$ 
or an isogeny $E^3 \rightarrow A$, the estimate for $A$ follows
using results in \S \ref{isogenies}); if $A$ contains at least four
elliptic curves then there is an isogeny as above, and the statement follows).

\bigskip \noindent
{\bf Remark.} When $A = J(C)$ is the Jacobian of a curve of genus $g>1$, 
there is an effective bound for the function $N_A(t)$
due to Kani (cf. \cite{K1}, Theorem 4), which is
of order $O(t^{2g^2-2})$. In particular, 
for $g=2$ of order $O(t^6)$, for $g=3$ of order $O(t^{16})$.
As the orders of bounding functions found in the present paper are quite smaller,
we are encouraged to believe that our approach may lead to some 
sharper effective bounds for arbitrary $g$.

\vfill

\noindent \small address:
Dipartimento di Matematica e Informatica, Universit\`a di Perugia, 
Via Vanvitelli 1, 06123 Perugia, Italia $-$ email: {\tt lucio.guerra@unipg.it}

\end{document}